\documentclass[10pt]{amsart}

\usepackage{amssymb,latexsym,amsmath}
\usepackage{graphics}                 
\usepackage{color}                    
\usepackage{hyperref}                 
\usepackage[all]{xy}

\begin{document}

\newtheorem{theorem}{Theorem}
\newtheorem{proposition}[theorem]{Proposition}
\newtheorem{lemma}[theorem]{Lemma}
\newtheorem{corollary}[theorem]{Corollary}
\newtheorem{question}[theorem]{Question}
\newtheorem{remark}[theorem]{Remark}
\newtheorem{example}[theorem]{Example}
\newtheorem{conjecture}[theorem]{Conjecture}
\newtheorem{definition}[theorem]{Definition}
\newtheorem{correction}{Correction}

\def\mA{{\mathcal{A}}}

\def\A{{\mathbb{A}}}
\def\C{{\mathbb{C}}}
\def\H{{\mathbb{H}}}
\def\L{{\mathbb{L}}}
\def\P{{\mathbb{P}}}
\def\Q{{\mathbb{Q}}}
\def\R{{\mathbb{R}}}
\def\Z{{\mathbb{Z}}}

\def\Ch{{\rm Ch}}
\def\id{{\rm id}}
\def\sp{{\rm SP}}

\def\mC{{\mathcal{C}}}
\def\mP{{\mathcal{P}}}
\def\mZ{{\mathcal{Z}}}

\title{The   Euler number of a $\C^*$-invariant subvariety  in $\P^n$}
\author{Wenchuan Hu}

\address{
School of Mathematics\\
Sichuan University\\
Chengdu 610064 \\
P. R. China
}

\email{huwenchuan@gmail.com} 

\begin{abstract}
 In this note we show that the Euler number of
 a projective variety $\C^*$-equivariantly  embedded into a projective space $\P^n$ is bounded by
 $n+1$, as  conjectured by Carrell and Sommese.
\end{abstract}

\maketitle
\pagestyle{myheadings}
 \markright{The Euler number of an equivariant embedded variety }


\section{Introduction}

If $X$ is a smooth complex projective variety admitting $\C^*$-action, then the topological information of $X$
can be recovered from that of the fixed point set (see \cite{Frankel}, \cite{Bialynicki-Birula}, \cite{Carrell-Sommese1}, etc.).

The point is that there exists a Morse-Bott function on $X$ whose critical points are exactly the fixed components of the action.
It has been shown by Blanchard that any projective manifold $X$ admitting a $\C^*$-action with nonempty fixed point set admits
an equivariant projective embedding.  If $X\subset \P^n$ is $\C^*$-equivariantly embedding and if $X$ is contained in no $\C^*$-invariant
hyperplane of $\P^n$, then there is a  Morse-Bott function $f$ on $\P^n$ satisfying $f(\P^n)=f(X)$ (see \cite{Carrell-Sommese2}), but it is not
true that  $X$ meets every component of $(\P^n)^{\C^*}$. One of  such examples is the $d$-uple Veroese embedding of $\P^m\to \P^{(^{m+d}_{~d})-1}$.
A natural question is whether there is a clear relation between the Euler number of $X$ and $\P^n$.

In this note we show the following result holds.

\begin{theorem}\label{Th01}
Let $X$ be a smooth connected complex projective variety admitting a $\C^*$-action.
 If $X^{\C^*}$ is finite and $X$ is equivariantly
embedded in  $\P^n$, then $$\chi(X)\leq \chi(\P^n),$$
where $\chi$ denotes the Euler number.
\end{theorem}

This was conjectured by Carrell and Sommese \cite[p.572]{Carrell-Sommese2}.
We also show by examples that if $X^{\C^*}$ is not isolated or $\P^n$ is replaced by a complex projective variety $M$, then the conclusion in Theorem \ref{Th01} fails.
\begin{remark}
Let $X$ be given in Theorem \ref{Th01}. Then  the lower bound of $X$ is  $\dim_{\C} X+1$ since the odd Betti numbers of $X$ vanishes and all the even Betti numbers are positive. The equality of the lower bound is reached for complex projective spaces.
\end{remark}

\section{Equivariant embedded into projective spaces}
Let $\P^n$ be the complex projective space of dimension $n$ and let $X$ be a complex projective variety.
A $\C^*$-action  on a complex projective variety $X$ means that there is
a holomorphic map $\mu:\C^*\times X\to X$ such that $\mu(1,x)=x$ and $\mu(t_1t_2,x)=\mu(t_1,\mu(t_2,x))$ for all $t_1,t_2\in\C^*$ and $x\in X$.
It is well known that any holomorphic action of $\C^*$ on $\P^n$ arise through a 1-parameter subgroup
$\lambda:\C^*\to \P GL(n,\C)$. Hence up to a projective transformation, a $\C^*$
action on $\P^n$ is of the form
\begin{equation}\label{eqn1}
\lambda\cdot [z_0:z_1:,...,z_n]=[\lambda^{a_0}z_0:\lambda^{a_1}z_1:,...,\lambda^{a_n}z_n],
\end{equation}
where $a_0,a_1,...,a_n$ are integers satisfying $a_0\leq a_1\leq ...\leq a_n$ (cf. \cite{Carrell}).

Suppose that $X$ is embedded equivariantly in $\P^n$, where the action of $\C^*$ on $\P^n$ is given
by Equation \eqref{eqn1}. Let $X^{\C^*}$ be the set of fixed points under the action of $\C^*$.

\begin{lemma}
One has $X^{\C^*}\subset (\P^n)^{\C^*}$.
\end{lemma}
\begin{proof}
It follows from the assumption that $X$ is equivariantly embedded in $\P^n$. If $x\in X$ is a point
such that $\lambda\cdot x:\mu(\lambda, x)=x$ for all $\lambda\in \C^*$, then $x$ is a fixed point
of the action and so $x\in (\P^n)^{\C^*}$.
\end{proof}

By the same reason, if $X$ and $M$ are smooth projective varieties admitting $\C^*$-actions
 and $X$ is $\C^*$ equivariantly embedded in $M$, then we have the following result.

\begin{corollary}
Let $X$ be $\C^*$ equivariantly embedded in $M$. If  $M^{\C^*}$ is  finite, then $\chi(X)\leq \chi(M)$.
\end{corollary}
\begin{proof}
Since $X^{\C^*}\subset M^{\C^*}$, we have $\# X^{\C^*}\leq \# M^{\C^*}$. By the fixed point theorem,
we have $\chi(X)=\# X^{\C^*}$ and $\chi(M)=\# M^{\C^*}$. This completes the proof of the Corollary.
\end{proof}

\begin{lemma}\label{lemma2}
Let $\mu:\C^*\times \C^n\to \C^n$ be the action given by $\mu(\lambda, (z_1,...,z_n))=(z_1,...,z_k,\lambda^{a_{k+1}}z_{k+1},...,\lambda^{a_n}z_n))$,
where $k\geq 1$ and $a_i\neq 0$ for all $i=k+1,...,n$. Suppose that $U\subset \C^n$ is  an irreducible affine variety equivariantly embedded in $\C^n$.
Then  the fixed point set $U^{\C^*}$ is irreducible. In particular, if $U^{\C^*}$ is nonempty and finite, then $U^{\C^*}$ contains exact one point.
\end{lemma}
\begin{proof}
It follows from a direct calculation that the fixed point set of the $\C^*$ action on $\C^n$ is $\C^k\subset \C^n$, spanned by the first
$k$-vector $e_1,...,e_k$, where $e_i=(0,..0,1,0,...,0)$ and $1$ is at the $i$-th position. This gives us an identification
$\C^n\cong \C^k\times \C^{n-k}$, where $\C^k$ is identified with the fixed point set of the action and $\C^{n-k}$ is identified with
$(\C^n)_p:=\{z\in \C^n|\lim_{\lambda\to 0}\lambda \cdot (z_1,..,z_n)=(z_1,...,z_k,0,...,0)\}$,
where $p=(z_1,...,z_k,0,...,0)\in (\C^n)^{\C^*}\cong \C^k$.  We denote by $\pi:\C^n\to \C^k $ the composed morphism
$\C^n\cong \C^k\times \C^{n-k}\stackrel{pr_1}{\longrightarrow} \C^k $.

Since $U$ is equivariantly embedded in $\C^n$, we have $U^{\C^*}\subset (\C^n)^{\C^*}$. At each $p\in U^{\C^*}$, the set
$$U_p=\{z\in U|\lim_{\lambda\to 0}\lambda \cdot z=p\}
$$
is contained in $(\C^n)_p$.

The restriction of the morphism $\pi$ on $U$ gives us a surjective morphism $\pi_{|U}:U\to U^{\C^*}$ whose fiber
is $U_p$. By assumption, $U$ is irreducible, we obtain that $U^{\C^*}$ is irreducible
 since a surjective morphism between quasi-projective varieties preserves irreducibility.
\end{proof}

Set $p_0=[1:0:...:0]$, $p_1=[0:1:...:0]$, ..., $p_n=[0:0:...:1]$. It is clear to see
that $p_0,p_1,...,p_n$ are fixed points of the action since
$$\lambda\cdot p_i=[0:...0:\lambda^{a_i}:0:...:0]=p_i.$$

Let $F_1,...,F_r$ be the connected components of $(\P^n)^{\C^*}$. A general result
says that each $F_i$ is a smooth projective subvariety (see \cite{Frankel}, \cite{Bialynicki-Birula}).
Moreover, in this case, each $F_i$ is a linear projective subspace of $\P^n$. Let $F_i^+=\{x\in \P^n|\lim_{\lambda \to 0}\lambda\cdot x\in F_i\}$,
then  $F_i^+$ is a vector bundle over $F_i$(see \cite{Bialynicki-Birula}).

Fix a hyperplane $\P^{m}\subset\P^{m+1}$ and a point $p\in \P^{m+1}-\P^m$.
For an algebraic set $V\subset \P^m$, the \emph{algebraic suspension} of $V$ with
vertex $p$ (i.e., cone over $V$)
is the set
$$\Sigma_p V:=\cup\{ l~|~ l \hbox{ is a projective line through $\P^0$ and intersects $V$}\}.
$$

\begin{proposition}\label{Prop4}
Let $X$ be a smooth connected projective variety  $\C^*$ equivariantly embedded in $\P^{n}$ and assume that
the $\C^*$ action on $\P^n$ is of the form
$$
\lambda\cdot [z_0:z_1:,...,z_n]=[\lambda^{a_0}z_0:\lambda^{a_1}z_1:,...,\lambda^{a_n}z_n],
$$
where $a_0,a_1,...,a_n$ are integers satisfying $a_0\leq a_1\leq ...\leq a_n$.
Suppose that $X^{\C^*}$ is finite and let $F_1,...,F_r$ be the connected component of $(\P^n)^{\C^*}$ for all $1\leq i\leq r$.
Then  $\#(X\cap F_1)\leq \dim F_1+1$, where $F_1$ is spanned by $p_0,...,p_{k}$.
\end{proposition}
\begin{proof}

Suppose $\#(X\cap F_1)> \dim F_1+1$  where $\dim F_1=k$.

First we assume that the points $X\cap F_1$ are in general positions in $F_1$ in the sense that $F_1$ is linearly spanned by
the points $X\cap F_1$.
Then one can make a suitable projective transformation $\sigma$ of $\P^n$ such that  $\sigma$ keeps $F_i$ fixed for $2\leq i\leq r$ and keeps $F_1$ invariant. Moreover, $\sigma$ moves points in $X\cap F_1$ to $p_0,...,p_{k}$.
Clearly, $\sigma(X)$ admits a $\C^*$-action and it is $\C^*$-equivariantly embedded into $\P^n$.  Since $\sigma(X)$ is isomorphic to $X$, $\chi(\sigma(X))=\chi(X)$. Hence we can assume that $p_0,...,p_{k}$ are in $X\cap F_1$.

Now we consider the projection  $pr_n:\P^n-p_n\to \P^{n-1} $ away from the point $p_n$. Set $X_n:=pr_n(X)$.

\textbf{Claim}: The image $X_n$ admits $\C^*$-action with finite fixed points and the embedding $X_n\subset\P^{n-1}$ is $\C^*$-equivariant.
\begin{proof}[Proof of the Claim]
For $\lambda\in \C^*$ and $p\in X\subset \P^n$, we write $p=[z_0:...:z_n]$ and $\lambda\cdot p=[\lambda^{a_0}z_0:...:\lambda^{a_n}z_n]\in X$
since $X$ is $\C^*$-invariant.
Then
$$
\begin{array}{ccl}
pr_n(\lambda\cdot p)
&=&pr_n([\lambda^{a_0}z_0:...:\lambda^{a_{n}}z_{n}])\\
&=&[\lambda^{a_0}z_0:...:\lambda^{a_{n-1}}z_{n-1}]\\
&=&\lambda\cdot ([z_0:...:z_{n-1}])\\
&=&\lambda\cdot pr_n([z_0:...:z_n])\\
&=&\lambda\cdot pr_n(p),
\end{array}
$$
i.e., the projection $pr_n$ commutes with the $\C^*$-action.  This implies that $X_n$ admits $\C^*$-action induced from $\P^{n-1}$. That is,
$X_n$ is $\C^*$-equivariantly embedded in $\P^{n-1}$.

For each $q\in X_n^{\C^*}$, the projective line $l_{qp_n}:=\overline{qp_n}$  connecting $q$ to $p_n$ is $\C^*$-invariant. Since $X$ is also
$\C^*$-invariant, the intersection $l_{qp_n}\cap X$ is $\C^*$-invariant. So if $l_{qp_n}$ does not lie in $X$, then $l_{qp_n}\cap X$ is
a finite point set and each point in $l_{qp_n}\cap X$ is a fixed point on $X$.  This shows that there is a map $f:X_n^{\C^*}\to S(X^{\C^*})$, where
$S(X^{\C^*})$ denotes the set of all subsets of $X^{\C^*}$.
 Moreover, if $q,q'\in X_n$ and $q\neq q'$, then $l_{qp_n}\cap l_{q'p_n}=p_n$. That is, $(l_{qp_n}-p_n)\cap (l_{q'p_n}-p_n)=\emptyset$.
 This shows that the  map $f$ is injective. So the set $\{q\in X_n^{\C^*}|l_{qp_n}\nsubseteq X \}$ is finite
since $X^{\C^*}$ is.
If $l_{qp_n}$ does lie in $X$, then  $q\in X$ and so $q\in X^{\C^*}$.
Since $X^{\C^*}$ is a finite set, $\{q\in X_n^{\C^*}|l_{qp_n}\subset X\}\subset X^{\C^*}$ is finite. In summary, $X_n^{\C^*}$ is finite.
This completes the proof of the Claim.
\end{proof}

Now we can keep projecting away from points $p_{n-1},...,p_{m+2}$, where $m=\dim X$. Let $X_{m+2}$ be the image of $X$ after  projections away from
$p_n,p_{n-1},...,p_{m+2}$. From the Claim, we know $X_{m+2}$ admits $\C^*$-action with finite fixed points and
the embedding $X_{m+2}\subset\P^{m+1}$ is $\C^*$-equivariant irreducible hypersuface. Clearly,  $p_j\in X_{m+2}$ for $0\leq j\leq k$ since
$p_j$ is fixed under the projections for $0\leq j\leq k$.

On one hand, since $F_1$ is fixed under the $\C^{*}$-action and
$X_{m+2}$ is $\C^*$-invariant, $F_1\cap X_{m+2}$ is fixed under $\C^*$-action. So we have $F_1\cap X_{m+2}\subset X_{m+2}^{\C^*}$.
It has been shown in the Claim that $X_{m+2}^{\C^*}$ is finite.
On the other hand, one has $0=\dim(F_1\cap X_{m+2})\geq \dim F_1+\dim X_{m+2}-(m+1)=k-1$ since $\dim F_1=k$.
Therefore, we get $k\leq 1$.

If $k=0$, then  $\dim F_1=0$ and the Proposition is trivial.

 If $k=1$, then $F_1$ is a projective line .
For any $l\subset \P^{m+1}$, if $l\nsubseteq X_{m+2}$,
then $l\cdot X_{m+2}=\deg X_{m+2}$. Now we consider the case that $F_1$ is a projective line in $\P^n$.
Since $F_1$ a $\C^*$-fixed, $F_1\nsubseteq X_{m+2}$.

Suppose the conclusion in the Proposition fails, that is, $F_1\cap X_{m+2}$ contains at least $\dim F_1+2=3$ points.
Let $q_0\in (\P^{m+1})^{\C^*}-F_1$ be a  point whose existence is clear unless $(\P^{m+1})^{\C^*}=F_1$. This is not
possible since $(\P^n)^{\C^*}$ is not connected unless the action of $\C^*$ is trivial.


For any $p\in F_1$, $l_{q_0p}$ is $\C^*$ invariant. Moreover, by the reason of degree,
 $l_{q_0p}\cap X_{m+2}$ contains at least one point $p'\in X_{m+2}^{\C^*}$ but $p'\neq q_0$ or $p$.
 Since   $l_{q_0p}\cap l_{q_0q}=q_0$ for any pair $p,q\in F_1$, this contradicts to the fact that $X_{m+2}^{\C^*}$ is finite.
 The contradiction comes from the hypothesis that $\#(X\cap F_1)>\dim F_1+1$. Therefore $\#(X\cap F_1)\leq \dim F_1+1$.

Now we consider the case that $X\cap F_1$ is not of  general positions.
Since $X^{\C^*}\subset (\P^n)^{\C^*}$ and $X\cap F_1\subset F_1$, we denote by $F_1'\subset F_1$ a linear projective subspace of minimal dimension
containing $X\cap F_1$. Set $k'=\dim F_1'$.
 Then the above arguments are applied to  $F_1'$, we get $\#(X\cap F_1')\leq \dim F_1'+1$. Since $\#(X\cap F_1')=\#(X\cap F_1)$ and
 $\dim F_1'< \dim F_1$, we obtain that $\#(X\cap F_1)\leq \dim F_1+1$. This completes the proof of proposition.
\end{proof}

Now we will show Theorem \ref{Th01} below.
\begin{proof}[Proof of Theorem \ref{Th01}]
The Lefschetz fixed point theorem says that if $Y$ is a complex projective variety admitting a $\C^*$ action with the
fixed point set $Y^{\C^*}$, then one has $\chi(Y)=\chi(Y^{\C^*})$.

Since $X$ is a smooth connected projective variety  $\C^*$ equivariantly embedded in $\P^{n}$
such that $X^{\C^*}$ is finite and let $F_1,...,F_r$ be the connected component of $(\P^n)^{\C^*}$.
From Proposition \ref{Prop4}, we have
$\#(X\cap F_1)\leq \dim F_1+1$. Since $F_1^+$ is an algebraic vector bundle over $F_1$, which is a linear $k$-plane in $\P^n$, one
has $\P^n-F_1^+\cong \P^m$, where $m=n-k-1$. Note that $F_2,...,F_r$ are the fixed component of the induced $\C^*$ action on $\P^m$.
By induction hypothesis, we have $\chi(X\cap \P^m)\leq m+1=n-k$.
Since $X^{\C^*}=\cup_{i=1}^r X\cap F_i$, we get
$\chi(X)=\chi(X\cap F_1^+)+\chi(X\cap \P^m)\leq (k+1)+(n-k)=n+1$ by Proposition \ref{Prop4} and the induction hypothesis.
That is, we have $\chi(X)\leq n+1$.
\end{proof}

\begin{remark}
From the proof of Theorem \ref{Th01}, we observe that the smoothness  of $X$ can be replaced by the
irreducibility of $X$.
\end{remark}

\section{Examples}

In this section we will deal with examples. For some examples, the assumption in Theorem \ref{Th01} are satisfied, while
for others the assumptions are not satisfied.

\begin{example}[Segre embedding]
Consider the Segre embedding
$$s:\P^m\times\P^n\to \P^{(m+1)(n+1)-1}$$
$$([x_0:...:x_m],[y_0:...:y_n])\mapsto [x_0y_0:...:x_iy_j:...:x_my_n].$$
Let $X:=s(\P^m\times \P^n)$ be the image of $s$.

The $\C^*$-action on $\P^m\times\P^n$ is given by $\lambda,([x_0:...:x_m],[y_0:...:y_n]) \mapsto([\lambda^{a_0}x_0:...:\lambda^{a_m}x_m],[\lambda^{b_0}y_0:...:\lambda^{b_n}y_n])$, where $a_i$ are different from each for all
$0\leq i\leq m$ and $b_j$ are different from each for all
$0\leq j\leq n$.  The $\C^*$-action on $\P^{(m+1)(n+1)-1}$ is given by $(\lambda, [z_{ij}])\mapsto [\lambda^{a_ib_j}z_{ij}]$.
In this case $X$ admits a $\C^*$-action  with isolated fixed point. Clearly, $\chi(X)=\chi(\P^m\times\P^n)=(m+1)(n+1)=\chi(\P^{(m+1)(n+1)-1})$.
\end{example}

\begin{example}[Veronese embedding]
Consider the Veronese embedding
$$v:\P^n \to\P^{(^{n+1}_{~d})-1}$$
$$([x_0:...:x_d])\mapsto [x_0^d:...:x_0^{i_0}\cdots x_n^{i_n}:...:x_n^d],$$
where $i_0+...+i_n=d$.
The $\C^*$-action on $\P^{(^{n+1}_{~d})-1}$ is given by $(\lambda, [z_I])\mapsto [\lambda^{a_I} z_I]$, where $I=(i_0,...,i_n)$ is
the multiple index. If $a_I$ are different from each other for $I=(d,0,...,0),(0,d,...,0),...,(0,...,0,d)$, then the induced $\C^*$
action on $X:=v(\P^n)$ is $\C^*$-equivariantly embedded in $\P^{(^{n+1}_{~d})-1}$. Moreover, the fixed point set $X^{\C^*}$ is finite.
Clearly, $\chi(X)=\chi(\P^n)=n+1<\chi(\P^{(^{n+1}_{~d})-1})=(^{n+1}_{~d})$.
\end{example}

\begin{example}[Pl\"{u}cker embedding]

The Pl\"{u}cker embedding is the map $\iota$ defined by

    \begin{align*} \iota \colon \mathbf{G}(k, \C^n) &{}\rightarrow \P(\wedge^k \C^n)\\ \operatorname{span}( v_1, \ldots, v_k ) &{}\mapsto \C( v_1 \wedge \cdots \wedge v_k )， \end{align*}
where $\mathbf{G}(k, \C^n)$ is the Grassmannian, i.e., the space of all $k$-dimensional subspaces of the $n$-dimensional complex vector space $\C^n$.
Let $\C^*$ act on $\C^n$ is given by $(\lambda, (x_1,...,x_n))\mapsto (\lambda x_1,...,\lambda^n x_n)$ and we simply denote it $\lambda\cdot x$, where
$x=(x_1,...,x_n)\in \C^n$.

The $\C^*$-action on $\P(\wedge^k \C^n)$ is given by $(\lambda, \C( v_1 \wedge \cdots \wedge v_k))\mapsto \C( \lambda \cdot v_1 \wedge \cdots \wedge \lambda \cdot v_k))$. The fixed point set of this action on $\P(\wedge^k \C^n)$ consists of $k$-dimensional coordinate planes, i.e.,$k$-planes spanned
by $e_i=(0,...,0,1,0,...,0)$, where $1$ is at the $i$-th position. The number of such planes in $\P(\wedge^k \C^n)$  is exactly $(^{n+1}_{k+1})$. Moveover, all of such planes lie in $\mathbf{G}(k, \C^n)$. Therefore, $\chi(\mathbf{G}(k, \C^n) )=(^{n+1}_{k+1})=\chi(\P(\wedge^k \C^n))$.

\end{example}

The following example shows that the finiteness of $X^{\C^*}$  is necessary.
\begin{example}
Let $Y\subset \P^3$ be a smooth surface of degree $d\geq 3$. Consider the Segre embedding
$$
S:\P^3\times\P^1\to \P^7, $$
$$([z_0:...:z_3],[s:t])\mapsto [sz_0:...:sz_3:tz_0:...:tz_3].
$$

Let $X$ be the image of $Y\times \P^1$ under the  Segre embedding and let the $\C^*$ action $\phi_{\lambda}$
on $\P^7$ is given by
$$
\Phi:\C^*\times \P^7\to \P^7, (\lambda,[y_0:...:y_7])\mapsto [y_0:...:y_3:\lambda y_4:...:\lambda y_7],
$$
where $\phi_{\lambda}(-)=\Phi(\lambda,-)$.
Then $\chi(X)>\chi(\P^7)$.
\end{example}

\begin{proof}
The fixed point set $(\P^7)^{\C^*}$ of the action $\phi$ is  $F_0\cup F_1$, where
$$F_0=\{(y_0=y_1=y_2=y_3=0) \}$$ and
$$F_1=\{(y_4=y_5=y_6=y_7=0)\},$$ each
of them is isomorphic to the projective space $\P^3$.
Let $\Psi$ be the $\C^*$-action on $\P^1$ given by  $\Psi(\lambda, [s:t])=[s:\lambda t]$ and set  $\psi_{\lambda}(-)=\Psi(\lambda,-)$.
Note that $\phi_{\lambda}$ preserves $X$ since $\phi_{\lambda}( S(z,\mu))=S(z,\psi_{\lambda}(\mu))$. To see this, let
$z=[z_0:...z_3]\in \P^3$ and $\mu=[s:t]\in \P^1$, then we have
$$
\begin{array}{ccl}
\Phi(\lambda, S([z_0:...z_3],[s:t]))
&=&\Phi(\lambda, [sz_0:...:sz_3:tz_0:...:tz_3]\\
&=&[sz_0:...:sz_3:\lambda tz_0:...:\lambda tz_3])\\
&=& S([z_0:...z_3],[s:\lambda t])\\
&=& S([z_0:...z_3],\psi_{\lambda}([s:t])).
\end{array}
$$

Hence $\Phi$ induces a $\C^*$ action on $X$ which is  denoted by $\Phi_{|X}$. The fixed point set $X^{\C^*}$ of $\Phi_{|X}$
is the intersection of $X$ and $(\P^7)^{\C^*}$, which is $(F_0\cap X)\cup (F_1\cap X)$. Note that both $F_0\cap X$ and $F_1\cap X$
are isomorphic to $Y$ and $(F_0\cap X)\cap (F_1\cap X)=\emptyset$.

Note that
\begin{equation}\label{eqn2}
\chi(\P^7)=8,
\end{equation}
while
\begin{equation}\label{eqn3}
\begin{array}{ccl}
\chi(X)
&=&\chi(Y\times \P^1)=2\chi(Y)\\
&=& 2d(6-4d+d^2).
\end{array}
\end{equation}

Therefore, if $d\geq 3$, then $\chi(X)\geq \chi(\P^7)$.
\end{proof}

Now we consider the relation of Euler numbers for a smooth complex projective  variety $\C^*$ equivariantly embedded
into another smooth complex projective  variety.
Let $X\subset M$ be a $\C^*$ equivariantly embedding. If $M$ is not the projective space, then the following
simple example shows that even if $X^{\C^*}$ is finite, one still has $\chi(X)>\chi(M)$.

\begin{example}Let $C$ be a smooth elliptic curve and let $M:=C\times\P^1$ be the product of $C$ and
$\P^1$.  Let $\mu:\C^*\times M\to M$ be the action given by
$$
\mu(\lambda, (c,[z_0:z_1]=(c,[z_0:\lambda z_1]).
$$
Fix a point $c_0\in C$ and let $X:=\P^1\hookrightarrow M$ be given by $[z_0:z_1]\mapsto (c_0,[z_0:z_1]$
Then $\P^1$ is $\C^*$ equivariantly embedding in $M$
and $\chi(M)=0<2=\chi(\P^1)$.
Moreover, the fixed point set of $\P^1$ contains two points
$[1:0]$ and $[0:1]$.
\end{example}

\section*{Acknowledgements}
I would like to thank Shouxin Dai and Baohua Fu for useful conversions after the first version of the manuscript. 
The project was partially sponsored by  STF of Sichuan province,China(2015JQ0007) and NSFC(11521061).


\begin{thebibliography}{AAAA}

\bibitem[B-B]{Bialynicki-Birula}
A. Bialynicki-Birula,
{\sl Some theorems on actions of algebraic groups.}
Ann. of Math. (2) 98 (1973), 480--497.

\bibitem[C]{Carrell}
J. B. Carrell,
\emph{Holomorphic $\C^*$ actions and vector fields on projective varieties}. Topics in the theory of algebraic groups, 1--37,
Notre Dame Math. Lectures, 10, Univ. Notre Dame Press, Notre Dame, IN, 1982.

\bibitem[CS1]{Carrell-Sommese1}J. B. Carrell and A. J. Sommese,
{\sl ${\C}^{\ast}$-actions. } Math. Scand. 43 (1978/79), no. 1, 49--59.

\bibitem[CS2]{Carrell-Sommese2} J. B. Carrell and A. J. Sommese,
{\sl Some topological aspects of $\C^{\ast} $ actions on compact Kaehler manifolds.}
Comment. Math. Helv. 54 (1979), no. 4, 567--582.

\bibitem[F]{Frankel}
T. Frankel,
{\sl  Fixed points and torsion on K\"ahler manifolds}
Ann. of Math. (2) 70 1959 1--8.
\end{thebibliography}
\end{document}